\documentclass[11pt]{article}
\usepackage{graphicx}    
\usepackage{amsmath}
\usepackage{amsthm}
\usepackage{amssymb}
\usepackage[alphabetic]{amsrefs}

\newtheorem{theorem}{Theorem}[section]
\newtheorem{lemma}[theorem]{Lemma}

\theoremstyle{definition}
\newtheorem{definition}{Definition}

\theoremstyle{remark}

\renewcommand{\phi}{\varphi}

\newcommand{\e}{\varepsilon}
\newcommand{\R}{\mathbb{R}}
\renewcommand{\S}{\mathbb{S}}

\newcommand{\close}[1]{\overline{#1}}

\renewcommand{\u}{\mathbf{u}}
\renewcommand{\v}{\mathbf{v}}
\newcommand{\z}{\mathbf{z}}
\newcommand{\w}{\mathbf{w}}
\newcommand{\Om}{\Omega}
\renewcommand{\S}{\mathbf{S}}

\title{Viscosity Solutions of Balanced Quasi-Monotone Fully Nonlinear Weakly Coupled Systems}
\author{Andreas Minne\\
Martin H. Str\"omqvist}
\begin{document}
\maketitle

\begin{abstract}
We introduce so called balanced quasi-monotone systems. These are systems $F(x,r,p,X)=(F_1(x,r,p,X),\ldots,F_m(x,r,p,X))$, where 
$x$ belongs to a domain $\Om$, $r=u(x)\in\R^m$, $p=Du(x)$ and $X=D^2u(x)$, that can be arranged into two categories that are mutually 
competitive but internally cooperative. More precisely, for all $i\neq j$ in the set $\{1,2,\ldots,m\}$, $F_j$ is monotone non-decreasing 
(non-increasing) in $r_i$ if and only if $F_i$ is monotone non-decreasing (non-increasing) in $r_j$ and $F_j$ is a monotone function in 
$r_i$. We prove the existence and uniqueness of viscosity solutions to systems of this type. For uniqueness 
we need to require that $F_j$ is monotone increasing in $r_j$, at an at least linear rate. 
This should be compared to the quasi-monotone systems studied by Ishii and Koike in \cite{MR1116855}, where it is assumed that 
$F(x,r,p,X)\ge F(x,s,p,X)$ if $r\le s$. 
\end{abstract}

\section{Introduction} 

While the theory for viscosity solutions of scalar equations is highly developed \cite{MR1462699}, \cite{MR1118699}, 
much remains to be done for systems of equations. The reason for this is that the theory of viscosity solutions in its very definition is depending on the maximum principle which in general does not hold for elliptic systems. Results in the variational and
potential theoretic setting have no counterpart in the fully nonlinear case, where viscosity methods are used, unless some restrictions on the operators are done to impose a maximum principle. This has been done in \cite{MR1116855} through the assumption of a quasi-monotonicity condition (see Section \ref{sec:prob}). In this paper we introduce so called super-sub and sub-super solutions and define solutions as functions that 
are both super-sub and sub-super solutions. This enables us to extend the theory in \cite{MR1116855} to encompass a larger class of equations.   
To show existence, we use Perron's method while uniqueness follows along the same lines as in \cite{MR1116855}. With this result, we can cope with, for instance, competitive systems that does not seem to have been covered before.

\section{Problem Setting and Assumptions}\label{sec:prob}
The systems we shall deal with are of the type 
\begin{equation}\label{sys}
\left\{
\begin{aligned}
&F_1(x,\u(x),Du_1(x),D^2u_1(x))=0\\
&F_2(x,\u(x),Du_2(x),D^2u_2(x))=0\\
&\vdots\\
&F_m(x,\u(x),Du_m(x),D^2u_m(x))=0,
\end{aligned}\right.
\end{equation}
where $x$ belongs to an open bounded subset $\Omega$ of $\R^n$ and $\u(x)=(u_1(x),\ldots,u_m(x))$. 
Systems of type \eqref{sys} are called weakly coupled since each $F_j$ depends only on the pointwise 
values of $u_k$ for $k\neq j$ and not on the derivatives of $u_k$. Each $F_j$ is assumed to be a continuous function 
$\Omega\times\R^m\times\R^n\times\S^n\mapsto\R$, where $\S^n$ denotes the space of all symmetric $n$ times $n$ matrices. However, this continuity assumption on $F$ could be relaxed by considering the semicontinuous envelopes of $F$ (see Definition \ref{def:semi}).
To our knowledge only systems of this type have been treated with viscosity methods in any greater generality thus far. 
A typical assumption on \eqref{sys} is degenerate ellipticity:

For each $j\in\{1,\ldots,m\}$ and all $(x,r,p)\in\Omega\times\R^m\times\R^n$, 
\[F_j(x,r,p,X)\ge F_j(x,r,p,Y),\]
whenever $X\in\S^n$, $Y\in \S^n$ and $X\le Y$. 

We will subdivide the $m$ equations into two categories. The equations within each category will be mutually cooperative, while the two categories will be competitive. Let $1\le m_1\le m$ be a an integer. We write any vector $\xi\in \R^m$ as $\xi = (r,s)$, where 
$r= (r_1,\ldots,r_{m_1})$ and $s = (s_{1},\ldots,s_{m_2})$, $m_1+m_2=m$. For a vector valued function $\u=(\u_1,\u_2):\Omega\times\Omega\to\R^{m_1}\times R^{m_2}$ we will use the notation $(\u_1,\u_2)=(u_{11},u_{12}\ldots,u_{1m_1},u_{21},u_{22},\ldots,u_{2m_2})$. The following condition will be referred to as \emph{balanced quasi-monotonicity}:
\begin{align}\label{mon1}
&\left\{\begin{aligned}
&F_j(x,r,s,p,X)\le F_j(x,r,\sigma,p,X),\quad\text{whenever }s\le \sigma,\; \text{ and }1\le j\le m_1,\\
&F_j(x,r,s,p,X)\ge F_j(x,\rho,s,p,X),\quad\text{whenever }r\le \rho,\;r_j=\rho_j \text{ and }1\le j\le m_1,\\
&\text{for all }(x,p,X)\in \Omega\times\R^n\times\S^n, \\
\end{aligned}\right.\\
&\left\{\begin{aligned}\label{mon2}
&F_{m_1+j}(x,r,s,p,X)\le F_{m_1+j}(x,\rho,s,p,X),\quad\text{whenever }r\le \rho,\; \text{ and }1\le j\le m_2,\\
&F_{m_1+j}(x,r,s,p,X)\ge F_{m_1+j}(x,r,\sigma,p,X),\quad\text{whenever }s\le \sigma,\;s_j=\sigma_j \text{ and }1\le j\le m_2,\\
&\text{for all }(x,p,X)\in \Omega\times\R^n\times\S^n.
\end{aligned}\right.
\end{align} 
For example, the system 
\begin{equation}\label{competitive}
\left\{\begin{aligned}
&-\Delta u+\lambda u+\alpha\max(u,v)-f=0\quad \text{in }\Om,\\
& -\Delta v+\lambda v+\beta\max(u,v)-g=0\quad \text{in }\Om, \\
\end{aligned}\right.
\end{equation}
where $\alpha$, $\beta$ and $\lambda$ are positive constants, is included in this setting with 
$m_1=m_2=1$. We will return to this system at the end of Section \ref{last}. 

The conditions \eqref{mon1}-\eqref{mon2} should be compared to the \emph{quasi-monotone} criterium in \cite{MR1116855},
\begin{equation}\label{monorig}\begin{aligned}
&F_j(x,\eta,p,X)\le F_j(x,\xi,p,X),\quad\text{whenever }\xi\le \eta,\; \xi_j=\eta_j,\; \xi\in\R^m, \eta\in\R^m,\\
&\text{for all }j\in\{1,\ldots,m\}\text{ and }(x,p,X)\in \Omega\times\R^n\times\S^n. 
\end{aligned}
\end{equation}
Note that \eqref{mon2} becomes void if $m_1=m$ and that \eqref{monorig} then coincides with \eqref{mon1}. 

%
 
Before giving the definition of a viscosity solution we recall the semicontinuous envelopes of a function, 
and the notion of touching from above and below.  
\begin{definition} \label{def:semi}
Let $f$ be a function $\Omega\mapsto\R$. 
The upper (resp. lower) semicontinuous envelope of $f$ is given by 
\[
f^*(x) = \limsup_{y\to x}f(x),\qquad (f_*(x) = \liminf_{y\to x}f(x)). 
\]
\end{definition}
For a vector valued function $\mathbf{f}:\Om\to\R^k$ we write $\mathbf{f}^*$ to denote $(f_{1}^*,\ldots, f_k^*)$ and similarly $\mathbf{f}_*$ 
for $(f_{1*},\ldots, f_{k*})$.
\begin{definition}
We shall say a function $\phi:\Omega\to\R$ touches $f$ from above at $x\in\Omega$ if $\phi(x)=f(x)$ 
and $\phi>f$ in $\mathcal{N}\setminus \{x\}$ for some open neighborhood $\mathcal{N}$ of $x$.
Similarly, $\phi$ is said to touch $f$ from below at $x$ if $\phi(x)=f(x)$
and $\phi<f$ in $\mathcal{N}\setminus \{x\}$ for some open neighborhood $\mathcal{N}$ of $x$.   
\end{definition}

Henceforth we assume that the system \eqref{sys} is degenerate elliptic and satisfies the condition of balanced monotonicity \eqref{mon1}-\eqref{mon2}. Viscosity solutions will be referred to simply as solutions. 
\begin{definition}[Super-sub Solution]
A bounded function $(\u_1,\u_2):\Omega\times\Om\to \R^{m_1}\times\R^{m_2}$ is said to be a \emph{super-sub solution} of \eqref{sys} 
if $\u_*^{\:\:*} = (\u_{1*},\u_2^{*})$ satisfies, for each $x\in\Omega$, 
\[
F_j(x,\u_*^{\:\:*}(x),D\phi(x),D^2\phi(x))\ge 0, \qquad\text{for all }1\le j\le m_1,
\]
whenever $\phi\in C^2(\Omega)$ touches $u_{1j*}$ from below at $x$ and 
\[
F_{m_1+j}(x,\u_*^{\:\:*}(x),D\phi(x),D^2\phi(x))\le 0, \qquad\text{for all }1\le j\le m_2,
\]
whenever $\phi\in C^2(\Omega)$ touches $u_{2j}^{*}$ from above at $x$. 
We shall also write 
\begin{align*}
&F_j(x,\u(x),Du_{1j}(x),D^2u_{1j}(x))\ge 0,\quad 1\le j\le m_1,\\
&F_{m_1+j}(x,\u(x),Du_{2j}(x),D^2u_{2j}(x))\le 0,\quad 1\le j\le m_2,
\end{align*}
when the above conditions are met. 
\end{definition}

\begin{definition}[Sub-super Solution]
A bounded function $(\u_1,\u_2):\Omega\times\Om\to \R^{m_1}\times\R^{m_2}$ is said to be a \emph{sub-super solution} of \eqref{sys} 
if $\u^*_{\:\:*} = (\u_1^{*},\u_{2*})$ satisfies, for each $x\in\Omega$,  
\[
F_j(x,\u^*_{\:\:*}(x),D\phi(x),D^2\phi(x))\le 0, \qquad\text{for all }1\le j\le m_1,
\]
whenever $\phi\in C^2(\Omega)$ touches $u_{1j}^{*}$ from above at $x$ 
and 
\[
F_{m_1+j}(x,\u^*_{\:\:*}(x),D\phi(x),D^2\phi(x))\ge 0, \qquad\text{for all }1\le j\le m_2,
\]
whenever $\phi\in C^2(\Omega)$ touches $u_{2j*}$ from below at $x$. 
We shall also write 
\begin{align*}
&F_j(x,\u(x),Du_{1j}(x),D^2u_{1j}(x))\le 0,\quad 1\le j\le m_1,\\
&F_{m_1+j}(x,\u(x),Du_{2j}(x),D^2u_{2j}(x))\ge 0,\quad 1\le j\le m_2,
\end{align*}
when the above conditions are met. 
\end{definition}
\begin{definition}[Solution]\label{def_sol}
A bounded function $(\u_1,\u_2):\Omega\times\Om\to \R^{m_1}\times\R^{m_2}$ is said to be a \emph{solution} of \eqref{sys} 
if it is both a super-sub solution and a sub-super solution. 
\end{definition}

This should be compared to the definition of solution from \cite{MR1116855} given below. A subsolution (supersolution) according to \cite{MR1116855} can be thought of as a sub-sub solution (super-super solution) in our setting. 
\begin{definition}[Viscosity Subsolution \cite{MR1116855}]
A bounded function $\u:\Omega\to \R^m$ is said to be a viscosity subsolution of \eqref{sys} 
if $\u^* = (u_1^*,\ldots,u_m^*)$ satisfies, for each $j\in\{1,\ldots,m\}$ and $x\in\Omega$, 
\[
F_j(x,\u^*(x),D\phi(x),D^2\phi(x))\le 0, 
\]
whenever $\phi\in C^2(\Omega)$ touches $u_j^*$ from above at $x$.   
\end{definition}
\begin{definition}[Viscosity Supersolution \cite{MR1116855}]
A bounded function $\u:\Omega\to \R^m$ is said to be a viscosity supersolution of \eqref{sys} 
if $\u_* = (u_{1*},\ldots,u_{m*})$ satisfies, for each $j\in\{1,\ldots,m\}$ and $x\in\Omega$,  
\[
F_j(x,\u_*(x),D\phi(x),D^2\phi(x))\ge 0, 
\]
whenever $\phi\in C^2(\Omega)$ touches $u_{j*}$ from below at $x$. 
\end{definition}
\begin{definition}[Viscosity Solution \cite{MR1116855}]\label{def_sol_IK}
A continuous function $\u:\Omega\to \R^m$ is said to be a viscosity solution of \eqref{sys}
if it is both a viscosity subsolution and a viscosity supersolution. 
\end{definition}
Under the assumptions of degenerate ellipticity, quasi-monotonicity and the existsence of a subsolution 
$\mathbf{f}$ and a supersolution $\mathbf{g}$ 
such that $\mathbf{f}\le \mathbf{g}$ in $\close{\Omega}$, 
Perron's method is used in \cite{MR1116855} to prove the existence of a solution of \eqref{sys}. The structural assumption of 
quasi-monotonicity comes from the fact that the maximum (minimum) of two subsolutions (supersolutions) needs 
to be a subsolution (supersolution). Here the maximum of two functions $\u$ and $\v:\R^n\mapsto\R^m$ is defined as 
$\mathbf{w}(x) = (w_1(x),\ldots,w_m(x))$, where $w_j(x) = \max\{u_j(x),v_j(x)\}$ and analogously for the minimum. 

Although our definition of solution is based on sub-super- and super-sub solutions, as opposed to pure sub- and supersolutions in \cite{MR1116855}, the definitions of solution \ref{def_sol} and \ref{def_sol_IK} are equivalent. The major difference is that, 
if $(\u_1,\u_2)$ and $(\v_1,\v_2)$ are both super-sub solutions, then so is 
$\mathbf{w}$ given by 
\[
(\w_1(x),\w_2(x)) = (\min\{\u_1(x),\v_1(x)\},\max\{\u_2(x),\v_2(x)\}),
\] 
under the assumption of balanced quasi-monotonicity \eqref{mon1}-\eqref{mon2}. An analogous statement holds for sub-super solutions, i.e.
\[
(\w_1(x),\w_2(x)) = (\max\{\u_1(x),\v_1(x)\},\min\{\u_2(x),\v_2(x)\}).
\] 
This is the key point that allows us to adapt the theory developed in \cite{MR1116855} to balanced quasi-monotone systems. 


\section{Existence}

We use the Perron method. The first step towards existence is the following lemma: 
\begin{lemma}\label{infsup}
Let $S$ be any non-empty bounded set of super-sub solutions of \eqref{sys}. 
Let $u_{1j}(x) = \inf_S \{v_{1j}(x):\v = (\v_1,\v_2)\in S\}$, $j=1,\ldots,m_1$, $u_{2j}(x) = \sup_S \{v_{2j}(x):\v=(\v_1,\v_2)\in S\}$, 
$j=1,\ldots,m_2$. If $u_{1j*}(x)>-\infty$ for $1\le j\le m_1$ and $u_{2j}^*(x)<\infty$ for $1\le j\le m_2$, then $\u = (\u_1,\u_2)$ is a super-sub solution of \eqref{sys}. 
\end{lemma}
\begin{proof}
We first show that 
\[
F_j(x,\u(x),Du_{1j}(x),D^2u_{1j}(x))\ge 0,\qquad j=1,\ldots,m_1. 
\]
Suppose $\phi\in C^2$ touches $u_{1j*}$ from below at $x\in \Omega$. By definition there is a sequence $\u^k\in S$, that depends on $j$, such that 
$\lim_{k\to\infty}u_{1j*}^k(x)= u_{1j*}(x)$. Since $\{\u^k(x)\}_k$ is bounded, there is a subsequence such that  
\[
\u_*^{k\:*}(x)\to (r,s),\qquad r_j =u_{1j*}(x).
\] 
By definition of $\u$, we have 
\begin{equation*}
r\ge \u_{1*}(x),\qquad s\le \u_2^{*}(x). 
\end{equation*}
Additionally, $\phi$ touches $u_{1j*}^k$ from below at $x^k$ and $x^k\to x$. Using the continuity of $F$ 
and \eqref{mon1}, we find 
\begin{align*}
0&\le \lim_k F_j(x^k,\u_{*}^{k\:*}(x^k),D\phi(x^k),D^2\phi(x^k))= F_j(x,r,s,D\phi(x),D^2\phi(x))\\
&\le F_j(x,\u_*^{\:\:*}(x),D\phi(x),D^2\phi(x)). 
\end{align*}

To prove 
\[
F_{m_1+j}(x,\u(x),Du_{2j}(x),D^2u_{2j}(x))\le 0,  \qquad 1\le j\le m_2,
\]
we assume $\phi\in C^2$ touches $u^{*}_{2j}$ from above at $x\in \Omega$. In analogy to the first case, we can produce a sequence $\u^k\in S$ 
such that 
\[
\u_*^{k\:*}(x)\to (r,s), \quad\text{where }r\ge \u_{1*}(x),\: s\le \u_2^{*}(x),\: s_j=u_{2j}^*(x), 
\]
and $\phi$ touches $u_{2j}^{k*}$ from above at $x^k$ and $x^k\to x$. 
Again by continuity and \eqref{mon2}, 
\begin{align*}
0&\ge \lim_k F_{m_1+j}(x^k,\u_*^{k\:*}(x^k),D\phi(x^k),D^2\phi(x^k))= F_{m_1+j}(x,(r,s),D\phi(x),D^2\phi(x))\\
&\ge F_{m_1+j}(x,\u_*^{\:\:*}(x),D\phi(x),D^2\phi(x)). 
\end{align*}

\end{proof}

\begin{theorem}[Existence]\label{thm:ex}
If there exist a bounded super-sub solution $\mathbf{z}=(\mathbf{z}_1,\mathbf{z}_2)$ and a bounded sub-super solution $\mathbf{w}=(\mathbf{w}_1,\mathbf{w}_2)$ of \eqref{sys} 
such that 
\[
\begin{array}{l}
\z_1\ge \w_1\text{ in }\Omega\\
\z_2\le \w_2\text{ in }\Omega,
\end{array}
\]
then there exists a solution $\u$ of \eqref{sys} such that 
\[
\begin{array}{l}
\z_1\ge \u_1\ge \w_1\text{ in }\Omega\\
\z_2\le \u_2\le \w_2\text{ in }\Omega.
\end{array}
\]
\end{theorem}
\begin{proof}
Consider the class 
\[
S = \{\text{super-sub solutions }\v=(\v_1,\v_2)\text{ of \eqref{sys} such that }\z_1\ge \v_1\ge \w_1\text{ and }\z_2\le \v_2\le \w_2\},  
\]
which by hypothesis is non empty. 
Let 
\[
\left\{\begin{array}{l}
u_{1j}(x) = \inf\{v_{1j}(x):\v\in S\},\qquad j=1,\ldots, m_1\\
u_{2j}(x) = \sup\{v_{2j}(x):\v\in S\},\qquad j=1,\ldots, m_2,
\end{array}\right.
\]
and set $\u=(\u_1,\u_2)$. 
Then $\u$ is a super-sub solution according to Lemma~\ref{infsup}. 
Arguing by contradiction, we will prove that $\u$ is also a sub-super solution. 
Assume $\phi\in C^2$ touches $u_{1j}^*$ from above at $x$ and that 
\begin{equation}\label{absurd}
F_j(x,\u^*_{\:\:*}(x),D\phi(x),D^2\phi(x))\ge \theta>0. 
\end{equation}
We will show that this implies 
\begin{equation}\label{strict}
w_{1j}^*(x)<\phi(x)=u_{1j}^*(x). 
\end{equation}
If not, $w_{1j}^*(x)=u_{1j}^*(x)$ from the definition of $S$, and $\phi$ touches also $w_{1j}^*$ from above at $x$. 
Since $\mathbf{w}$ is a sub-super solution, 
\begin{align*}
0\ge F_j(x,\mathbf{w}^*_{\:\:*}(x),D\phi(x),D^2\phi(x)) 
\ge F_j(x,\u^*_{\:\:*}(x),D\phi(x),D^2\phi(x))>0,  
\end{align*}
by definition of $\u$ and \eqref{mon1}, a contradiction. 
Set $\tilde \u = (\tilde\u_1,\u_2)$, where 
\[
\tilde u_{1i} = \left\{\begin{array}{l}
u_{1i}\quad\text{if }i\neq j,\\
\phi\quad\text{if }i=j.
\end{array}\right.
\]
Then 
\begin{equation*}
F_j(x,\tilde \u^*_{\:\:*}(x),D\phi(x),D^2\phi(x)) = \theta>0.
\end{equation*}  
We claim that in a sufficiently small ball $B(x,\delta)$ with radius $\delta$ centered at $x$, 
\begin{equation}\label{loc}
F_j(y,\tilde \u^*_{\:\:*}(y),D\phi(y),D^2\phi(y)) \ge \theta/2,\qquad y\in B(x,\delta). 
\end{equation}
If not, there is a sequence $x^k\to x$ along which 
\[
\lim_{k\to\infty}F_j(x^k,\tilde \u^*_{\:\:*}(x^k),D\phi(x^k),D^2\phi(x^k)) < \theta/2.
\] 
For a subsequence there holds 
$\tilde\u^*_{\:\:*}(x^k)\to (r,s)$. By definition, $\tilde\u_1^*$ is upper semicontinuous and $\tilde \u_{2*}$ is lower semicontinuous, 
so $\tilde\u_1^*(x)\ge r$, $\tilde{\u}_{1j}^*(x)=\phi(x)$, and $\tilde \u_{2*}(x)\le s$. 
Thus 
\begin{align*}
&\lim_kF_j(x^k,\tilde \u^*_{\:\:*}(x^k),D\phi(x^k),D^2\phi(x^k)) = F_j(x,r,s,D\phi(x),D^2\phi(x))\\
&\ge F_j(x,\tilde \u^*_{\:\:*}(x),D\phi(x),D^2\phi(x)) = \theta,  
\end{align*}
by \eqref{mon1}, a contradiction. 

By \eqref{loc} and \eqref{mon1} we have 
\begin{equation}\label{theta/2}
F_j(y,\tilde \u^{\:\:*}_*,D\phi(y),D^2\phi(y))\ge \theta/2,\quad y\in B(x,\delta). 
\end{equation}
Note that we are now considering $\tilde\u^{\:\:*}_*$, not $\tilde\u^*_{\:\:*}$.
After shrinking $\delta$ if necessary, we can find $\e>0$ such that 
\[
w_{1j}^*(y)\le \phi(y)-\e, \quad y\in B(x,\delta). 
\]
This is a consequence of \eqref{strict} and upper semicontinuity. 
Replacing $\delta$ and $\e$ by smaller numbers if needed, we claim that  
\begin{equation}\label{halfball}
u_{1j}^*(y)\le \phi(y)-\e,\quad  y\in B(x,\delta)\setminus\close{B(x,\delta/2)}. 
\end{equation}
Since $\phi$ touches $u_{1j}^*$ from above at $x$, $\phi>u_{1j}^*$ in $\close{B(x,\delta)}\setminus \{x\}$ for small enough $\delta$. 
The claim now follows from upper semicontinuity. 
Redefine the $j$:th component of $\tilde \u_1$ by $\tilde u_{1j}=\phi(y)-\e$. 
If $\e$ is sufficiently small, the continuity of $F_j$ and \eqref{theta/2} tells us that 
\begin{equation}\label{pert}
F_j(y,\tilde \u^{\:\:*}_*,D\phi(y),D^2\phi(y)) \ge \theta/4,\quad y\in B(x,\delta).  
\end{equation}
Consider now the scalar equation 
\begin{equation}\label{scalar}
G(y,v(y),Dv(y),D^2v(y))=0\quad\text{in }B(x,\delta),  
\end{equation}
where 
\[
G(y,t,p,X) = F_j(y,u_{11*}(y),\ldots,u_{1(j-1)*}(y),t,u_{1(j+1)*}(y),\ldots,u_{1m_1*}(y),\u_2^*(y),p,X). 
\]
We will show that $v = \phi-\e$ is a super solution of \eqref{scalar}, in the usual sense of viscosity solutions of scalar 
equations, cf. \cite{MR1118699}. 
Suppose $\psi$ touches $v^*=v=\phi-\e$ from below at $y_0\in B(x,\delta)$. 
Then $D\psi(y_0)=D\phi(y_0)$ and $D^2\psi(y_0)\le D^2\phi(y_0)$.  
Using degenerate ellipticity and \eqref{pert}, we find that 
\[
G(y_0,v(y_0),D\psi(y_0),D^2\psi(y_0))\ge \theta/4. 
\]
Since $(\u_1,\u_2)$ is a super-sub solution of \eqref{sys}, $u_{1j}$ is clearly a super solution of \eqref{scalar}. 
According to the theory of scalar equations, $\gamma = \min(u_{1j},\phi-\e)$ is a super solution of 
\eqref{scalar}. Leaving the scalar equation, we show that $\hat \u=(\hat \u_1,\u_2)$, where 
\[
\hat u_{1i} = \left\{\begin{array}{rl}
u_{1i}&\text{if }i\neq j,\\
\gamma&\text{if }i=j,
\end{array}\right.
\]
is a super-sub solution of \eqref{sys} 
in $B(x,\delta)$. It is already clear that $F_j(y,\hat\u(y),D\hat u_{1j}(y),D^2\hat u_{1j}(y))\ge 0$ in the viscosity sense. 
Assume $\psi$ touches $u_{2i}^*$ from above at $y_0\in B(x,\delta)$, $1\le i\le m_2$. The facts that $(\u_1,\u_2)$ is a 
super-sub solution, $\gamma\le u_{1j}$ and \eqref{mon2} lead to 
\[
F_{m_1+i}(y_0,\hat \u_{1*}(y_0),\u_2^*(y_0),D\psi(y_0),D^2\psi(y_0))\le F_{m_1+i}(y_0,\u_{1*}(y_0),\u_2^*(y_0),D\psi(y_0),D^2\psi(y_0))\le 0. 
\]
From \eqref{halfball} it is seen that $\gamma = u_{1j}$ outside $B(x,\delta/2)$. Thus the extension  
\[
\bar \u = \left\{\begin{array}{ll}
(\hat\u_1,\u_2)&\text{in }B(x,\delta),\\
(\u_1,\u_2)&\text{in }\Omega\setminus \close{B(x,\delta)}
\end{array}\right.
\]
is a super-sub solution of \eqref{sys} in $\Omega$. 
But since $u_{1j}^*(x) = \phi(x)$, there exists a point $x_0\in B(x,\delta)$ where $\phi(x_0)-\e<u_{1j}(x_0)$, 
i.e. $\bar u_{1j}(x_0)<u_{1j}(x_0)$. This contradicts the minimality of $u_{1j}$ and proves that \eqref{absurd} 
is false. In a completely analagous way, it can be shown that if $\phi$ touches $u_{2i*}$ from below, then  
\[
F_{m_1+i}(x,\u^*_{\:\:*}(x),D\phi(x),D^2\phi(x))\ge 0. 
\] 
The proof is thereby complete. 
\end{proof}

\section{Comparison and Uniqueness}\label{last}

In this section we study \eqref{sys} with Dirichlet data on $\partial\Om$. That is 
\begin{equation}\label{sysDP}
\left\{
\begin{aligned}
&F_j(x,\u(x),Du_j(x),D^2u_j(x))=0\qquad\text{in }\Om,\qquad 1\le j\le m,\\
&\u=0\qquad\text{on }\partial\Om.
\end{aligned}\right.
\end{equation}
As before we assume \eqref{mon1}-\eqref{mon2}. 
It is noteworthy that uniqueness can be proved without reference to comparison here. 
In addition to degenerate ellipticity, the following two assumptions (c.f. \cite{MR1116855}) are enough to prove uniqueness for \eqref{sysDP}: 
\begin{itemize}
\item[i)] There is a positive number $\lambda>0$ such that if $\xi,\eta\in\R^m$, $\xi\neq\eta$ and 
$\xi_j-\eta_j=\max_{1\le k\le m}|\xi_k-\eta_k|$, then 
\[
F_j(x,\xi,p,X)-F_j(x,\eta,p,X)\ge \lambda(\xi_j-\eta_j), \qquad\text{for all }(x,p,X). 
\]
\item[ii)] There is a continuous function $\omega:[0,\infty)\to[0,\infty)$ with $\omega(0)=0$ 
such that if $X,Y\in\S^n$, $\alpha>0$ and 
\[
-3\alpha\begin{pmatrix}I&0\\0&I\end{pmatrix}
\le \begin{pmatrix}X&0\\0&Y\end{pmatrix}
\le 3\alpha\begin{pmatrix}I&-I\\-I&I\end{pmatrix},
\] 
then 
\[
F_j(y,\xi,\alpha(x-y),-Y)-F_j(x,\xi,\alpha(x-y),X)\le \omega(\alpha|x-y|^2+1/\alpha),
\]
for all $1\le j\le m$, $x,y\in\Om$ and $\xi\in\R^m$. 
\end{itemize}
The proof (cf.\cite{MR1116855} Theorem 4.1) does not take into account the dependence of $F_j$ on $\xi_i$ for 
$i\neq j$. Thus 
\[
\emph{a solution to \eqref{sysDP} is unique if }\text{i)}\emph{ and }\text{ii)}\emph{ hold.}
\]

In our setting we have a comparison principle for sub-super and super-sub solutions. 
The proof requires a somewhat stronger assumption than i). 
\begin{itemize}
\item[i')]
Suppose $(r,s), (\rho,\sigma)\in \R^{m_1}\times\R^{m_2}$ and 
\[
\max\left\{\max_{1\le j\le m_1}(r_{j}-\rho_{j}), \max_{1\le j\le m_2}(\sigma_{j}-s_j)\right\}=\theta>0. 
\]
Then there is a $\lambda>0$ such that 
\[
\left\{\begin{array}{ll}
F_j(x,r,s,p,X)-F_j(x,\rho,\sigma,p,X)\ge \lambda(r_{j}-\rho_{j})&\text{if }\theta = r_{j}-\rho_{j},\\
F_{m_1+j}(x,\rho,\sigma,p,X)-F_{m_1+j}(x,r,s,p,X)\ge \lambda(\sigma_{j}-s_{j})&\text{if }\theta = \sigma_{j}-s_{j}. 
\end{array}\right.
\] 
\end{itemize}
\begin{theorem}\label{comp}
Suppose i') and ii) hold. Then if $(\u_1,\u_2)$ is a sub-super solution and $(\v_1,\v_2)$ is a super-sub solution such that 
\[
\u_1^*\le\v_{1*} \text{  on }\partial\Om,\qquad\v_2^*\le\u_{2*} \text{  on }\partial\Om,
\]
then 
\[
\u_1^*\le\v_{1*} \text{  in }\Om,\qquad\v_2^*\le\u_{2*} \text{  in }\Om.
\]
\end{theorem}

We shall not give the proof of Theorem \ref{comp} since it closely follows that of Theorem 4.7 in \cite{MR1116855}. 
However, a few differences should be pointed out. The proof of \cite{MR1116855} is given for a subsolution $\u$ and a supersolution 
$\v$. The argument is by contradiction, assuming 
\begin{equation}\label{theta}
\max_{1\le k\le m,\;x\in \Om}(u_k^*(x)-v_{k*}(x))=\theta>0. 
\end{equation} 
Then by studying the local maxima of 
\[
\psi(k,x,y) = u_{k}^*(x)-v_{k*}(y)-\frac{\alpha}{2}|x-y|^2,\quad 1\le k\le m,\;x,y\in \Om,
\]
as $\alpha\to\infty$, the authors are able to derive a contradiction to \eqref{theta}. 
For the proof of Theorem \ref{comp} we assume that 
\begin{equation}\label{theta2}
\max\left\{\max_{1\le k\le m_1,\;x\in \Om}(u_{1k}^*(x)-v_{1k*}(x)),\;\max_{1\le k\le m_2,\;x\in \Om}(v_{2k}^*(x)-u_{2k*}(x))\right\}=\theta>0, 
\end{equation} 
and define $\psi(j,k_j,x,y)=\psi_j(k_j,x,y)$, for $j\in \{1,2\},\;1\le k_j\le m_j,\;x,y\in \Om$, where 
\[
\psi_1(k,x,y) = u_{1k}^*(x)-v_{1k*}(y)-\frac{\alpha}{2}|x-y|^2,\quad 1\le k\le m_1,\; x,y\in \Om,
\]
\[
\psi_2(k,x,y) = v_{2k}^*(x)-u_{2k*}(y)-\frac{\alpha}{2}|x-y|^2,\quad 1\le k\le m_2,\; x,y\in \Om.
\]
Then the proof of Theorem \ref{comp} is analogous to that of Theorem 4.7 in \cite{MR1116855}.

We conclude by elaborating on example \eqref{competitive} introduced in Section \ref{sec:prob}.  
\begin{equation*}
\left\{\begin{aligned}
&-\Delta u+\lambda u+\alpha\max(u,v)-f=0\quad \text{in }\Om,\\
& -\Delta v+\lambda v+\beta\max(u,v)-g=0\quad \text{in }\Om, \\
&u=v=0\quad\text{on }\partial\Om, 
\end{aligned}\right.
\end{equation*}
where $\alpha,\beta,\lambda$ are all positive, $f$ and $g$ are non-negative smooth functions, and $\Omega$ is a bounded smooth domain in $\R^n$. 
We will construct a super-sub solution to \eqref{competitive}. 
Let 
$\overline u$ be the solution to 
\begin{equation*}
\left\{\begin{aligned}
&-\Delta u+\lambda u-f=0\quad \text{in }\Om,\\
&u=0\quad\text{on }\partial\Om, 
\end{aligned}\right.
\end{equation*}
then $\overline u\ge 0$ by the maximum principle. 
Furthermore, let $\underline v$ solve 
\begin{equation*}
\left\{\begin{aligned}
&-\Delta v+\lambda v+\beta\max(\overline{u},v)-g=0\quad \text{in }\Om,\\
&v=0\quad\text{on }\partial\Om. 
\end{aligned}\right.
\end{equation*}
Then one can check that $(\overline u,\underline v)$ is a super-sub solution of \eqref{competitive}. 
Similarly we can produce a sub-super solution 
$(\underline u,\overline v)$ by letting $\overline v$ be the solution to 
\begin{equation*}
\left\{\begin{aligned}
&-\Delta v+\lambda v-g=0\quad \text{in }\Om,\\
&v=0\quad\text{on }\partial\Om.  
\end{aligned}\right.
\end{equation*}
As in the case of  $\overline u$, $\overline v\ge0$. 
Then choosing $\underline u$ to be the solution to 
\begin{equation*}
\left\{\begin{aligned}
&-\Delta u+\lambda u+\alpha\max(u,\overline{v})-f=0\quad \text{in }\Om,\\
&u=0\quad\text{on }\partial\Om,  
\end{aligned}\right.
\end{equation*}
we see that  $(\underline u,\overline v)$ is a sub-super solution. 
To see that $\underline u\le \overline u$ and $\underline v\le \overline v$ in $\Om$, the maximum principle is applied once more. 
Indeed, 
\[
-\Delta (\overline u-\underline u)+\lambda(\overline u-\underline u) = -\alpha\max(\overline u,\underline u)\le 0,
\] 
and $\overline u-\underline u=0$ on $\partial\Om$, so $\overline u\ge \underline u$ in $\Om$. 
Similarly we find that $\overline v\ge \underline v$.   
Now Theorem \ref{thm:ex} can be applied to infer the existence of a solution to \eqref{competitive}, 
which is unique by Theorem \ref{comp}.

\bibliographystyle{amsplain}
\bibliography{references}

\end{document}